\documentclass[11pt,reqno]{amsart}

\usepackage[T1]{fontenc}
\usepackage{amsmath,amssymb,amsthm,mathtools}
\usepackage{tikz-cd}
\usepackage{upref}
\usepackage{url}
\usepackage[colorlinks=true,linkcolor=blue,citecolor=blue,urlcolor=cyan]{hyperref}

\newtheorem{theorem}{Theorem}[section]
\newtheorem{lemma}[theorem]{Lemma}
\newtheorem{remark}[theorem]{Remark}
\newtheorem{proposition}[theorem]{Proposition}
\newtheorem{corollary}[theorem]{Corollary}
\newtheorem{conj}[theorem]{Conjecture}
\theoremstyle{definition}

\newcommand{\CC}{\mathbb{C}}
\newcommand{\QQ}{\mathbb{Q}}
\newcommand{\ZZ}{\mathbb{Z}}

\newcommand{\cC}{\mathcal{C}}
\newcommand{\cF}{\mathcal{F}}

\newcommand{\cS}{\mathcal{S}}

\newcommand{\ov}{\overline}

\newcommand{\alg}{\mathrm{alg}}
\newcommand{\homol}{\mathrm{hom}}
\newcommand{\tor}{\mathrm{tor}}
\DeclareMathOperator{\CH}{CH}
\DeclareMathOperator{\Pic}{Pic}
\DeclareMathOperator{\Gr}{Gr}
\DeclareMathOperator{\im}{Im}
\DeclareMathOperator{\cl}{cl}
\DeclareMathOperator{\AJ}{AJ}

\title[Rationally connectedness and Generalized Franchetta conjecture]
{Rationally connectedness and O'Grady's generalized Franchetta conjecture for K3 surfaces}
\date{}
\author[Yuan Lu]{Yuan Lu}
\address{Department of Mathematics, ETH Zürich}
\email{yuan.lu@math.ethz.ch}

\begin{document}
\allowdisplaybreaks

\begin{abstract}
O'Grady's generalized Franchetta conjecture asks whether any codimension two cycle on the universal polarized K3 surfaces restricts to a multiple of the Beauville–Voisin class on a given K3 surface. We apply Bergeron–Li's result on the cohomology of universal K3 surfaces to give an affirmative answer to this conjecture when the universal K3 surface is rationally connected, including the case of genus $22$.
\end{abstract}

\maketitle
\tableofcontents

\section{Introduction}
\label{sec:introduction}

Throughout, we work over the field of complex numbers. Recall that for a projective K3 surface $S$, Beauville and Voisin showed in \cite{BV} that there exists a canonical class $o_S$ in the Chow group $\CH^2(S)$ that satisfies the following properties:
\begin{enumerate}
    \item For each closed point $x\in S$ on a (possibly singular) rational curve $C\subset S$, we have $[x]=o_S\in \CH^2(S)$.
    \item For any divisors $D_1,D_2\in \CH^1(S)$, the intersection $D_1\cdot D_2$ is a multiple of $o_S$.
    \item The second Chern class $c_2(S)=24o_S$.
\end{enumerate}
The class $o_S$ is called the Beauville--Voisin class of the K3 surface $S$.\\
\indent For each $g\ge 3$, denote $\cF_g$ the moduli space of primitively polarized K3 surfaces of degree $2g-2$ (or equivalently, genus $g$). Then we can take $\cF_g'\subset\cF_g$ as the smooth dense open subset parametrizing primitively polarized K3 surfaces with trivial automorphism groups. Hence, we have a universal family of K3 surfaces $\pi: \cS_g'\to \cF_g'$.\\
\indent Motivated by the Franchetta conjecture on the moduli space of curves \cite{AC87}, O'Grady presented the generalized Franchetta conjecture on the moduli space of K3 surfaces in \cite[Section 5]{OG}, which is stated as follows.
\begin{conj}[Generalized Franchetta Conjecture]
\label{Generalized Franchetta Conjecture}
     Assume $g\ge3$. For each closed point $x\in \cF_g'$, denote $S_x:=\pi^{-1}(x)$. Then, for any class $a\in \CH^2(\cS_g')$ and any closed point $x\in \cF_g'$, the restriction~$a|_{S_x}$ is a multiple of the Beauville--Voisin class $o_{S_x}$.
\end{conj}
At present, Conjecture \ref{Generalized Franchetta Conjecture} is largely open. So far, it is known in the following cases: in \cite{PSY}, Pavic--Shen--Yin proved Conjecture \ref{Generalized Franchetta Conjecture} for $3\le g\le 10$ and $g=12,13,16,18,20$, which are the genera $g\ge3$ with known Mukai models. Here, a Mukai model refers to a way to describe a general genus $g$ polarized K3 surface as the zero locus of a general section of some globally generated homogeneous vector bundle over a homogeneous variety. In \cite{FL}, Fu--Laterveer used the geometry of special cubic fourfolds to prove Conjecture \ref{Generalized Franchetta Conjecture} for $g=14$. In \cite{Lu}, Lu used Mukai's program in genus $11$ to relate the geometry of $\cS_{11}'$ with geometry of universal curve of genus $11$, and proved Conjecture \ref{Generalized Franchetta Conjecture} for $g=11$.\\
\indent The main result of this paper is the following theorem.
\begin{theorem}
\label{thm:main}
Suppose that $g\ge3$ and $\cS_g'$ is rationally connected. Then the generalized Franchetta conjecture holds for genus $g$.
\end{theorem}
Here, a quasi-projective variety is rationally connected if it is birational to a smooth projective variety that is rationally connected.

In genus $22$, Farkas and Verra proved that the universal K3 surface is unirational \cite[Theorem~1.1]{FV22}. Since a smooth projective unirational variety is rationally connected, we obtain the following corollary, which proves a new case to the generalized Franchetta conjecture.
\begin{corollary}
\label{cor:genus22}
Conjecture \ref{Generalized Franchetta Conjecture} holds for $g=22$.
\end{corollary}
\begin{remark}
For $3\le g\le10$ and $g=12,13,16,18,20$, Mukai models give unirational
parametrizations of the universal K3 surfaces \cite[Section 2]{PSY}. For $g=14$, the unirationality of the universal K3 surface follows from
\cite[Theorem~1.1]{FV14}. For $g=11$, Mukai's program \cite{Muk} induces a birational map between a projective bundle on $\cS_{11}'$ and $\cC_{11}$, the universal family of smooth curves of genus $11$, and $\cC_{11}$ is unirational by \cite[Corollary 2]{Bar}. Therefore, Theorem \ref{thm:main} recovers the
previously known cases of Conjecture \ref{Generalized Franchetta Conjecture}.
\end{remark}
The proof of Theorem \ref{thm:main} relies on the result of Bergeron–Li \cite{BL} on the cohomology of $\cF_g'$ and $\cS_g'$. In \cite[Theorem 1.2.1]{BL}, they proved the following cohomological version of the generalized Franchetta conjecture.
\begin{theorem}
\label{thm: cohomological GFC}
Assume $g\ge3$. For every $\alpha\in\CH^2(\cS_g')$, there exists
$t\in\QQ$ and a finite union $D\subset\cF_g'$ of Noether--Lefschetz
divisors such that, setting $V:=\cF_g'\setminus D$ and
$\cS_V:=\pi^{-1}(V)$, one has
\[
 \cl_{\cS_V}\bigl((\alpha-tc_2(T_\pi))|_{\cS_V}\bigr)
 =0\in H^4(\cS_V,\QQ).
\]
\end{theorem}
Under the assumption that $\cS_g'$ is rationally connected, we could control $\CH^2(\cS_g')$ using its cohomology group, and Theorem \ref{thm:main} would follow from Theorem \ref{thm: cohomological GFC}.\\
\indent This article is organized as follows. In Section \ref{sec:vanishing},
we generalize a result of Bloch--Srinivas and Murre on the second Chow group to quasi-projective varieties. In Section \ref{sec:franchetta}, we apply the result in Section \ref{sec:vanishing} to $\cS_g'$, and combine Bergeron–Li's result to give a proof to Theorem \ref{thm:main}.

\indent \textbf{Acknowledgments} The author is deeply grateful to his undergraduate advisor, Qizheng Yin, for suggesting the topic of the generalized Franchetta conjecture, for verifying this paper, and for his invaluable guidance and insight throughout the author's undergraduate study. His encouragement was essential to this work.

The author also thanks Lie Fu, Davesh Maulik, Weimufei Wu for useful discussions. The author is especially grateful for a discussion with Zhiyu Tian, which led to the main idea of this paper. The author is also especially grateful to Kaiyuan Gu, who helped the author verify the proof of Lemma \ref{lem:universal-weight}.

\indent Finally, the author is truly indebted to his parents and family for their unwavering support throughout his studies.

\indent \textbf{AI disclosure} The author used Rethlas \cite{Rethlas} in this paper to complete the details of the proof. The author supplied the main idea of using rationally connectedness and Bergeron–Li's results to Rethlas, and corrected the proof it generated. All the main ideas and strategies in this paper were developed by the author before the generation of Rethlas. All arguments appearing in the paper were fully understood, completely rewritten, and independently verified by the author.

\section{A result on the second Chow group}
\label{sec:vanishing}

The main result of this section is Proposition~\ref{prop:open-vanishing},
which is a generalization of the Bloch--Srinivas and Murre's result on the second Chow group to quasi-projective varieties. First, we prove a lifting lemma.

\textbf{Convention.} Throughout this paper, $\CH^*(X)_{\hom}$ is the group of integrally homologically trivial cycles.

\begin{lemma}
\label{lem:lifting}
Let $X$ be a smooth quasi-projective variety, and let
$j:X\hookrightarrow\ov X$ be a smooth projective compactification such
that $D:=\ov X\setminus X$ is a simple normal crossings divisor.
Then for every homologically trivial cycle $\alpha\in\CH^2(X)_{\homol}$, there exist an integer $N>0$
and a homologically trivial cycle $\ov\alpha\in\CH^2(\ov X)_{\homol}$ such that
\[
 j^*\ov\alpha=N\alpha.
\]
\end{lemma}

\begin{proof}
Write $D=\bigcup_{i\in I}D_i$ as its irreducible components, and let
$\iota_i:D_i\hookrightarrow\ov X$ be the inclusions. Each $D_i$ is
smooth and projective. Deligne's weight spectral sequence
\cite{Del} gives the exact sequence
\begin{equation}
\label{eq:weight-four}
 \bigoplus_{i\in I}H^2(D_i,\QQ)(-1)
 \xrightarrow{\ f\ }H^4(\ov X,\QQ)
 \longrightarrow\Gr^W_4H^4(X,\QQ)\longrightarrow0,
\end{equation}
where $f=\sum_i(\iota_i)_*$ is the sum of the Gysin maps, and $\Gr^W_4H^4(X,\QQ)$ is the graded piece of the weight filtration on $H^*(X,\QQ)$.

The localization exact sequence for Chow groups allows us to choose
$\alpha'\in\CH^2(\ov X)$ with $j^*\alpha'=\alpha$. Put
$c:=\cl_{\ov X}(\alpha')\in H^4(\ov X,\QQ)$. Since $\alpha$ is
homologically trivial, $j^*c=0$ in $H^4(X,\QQ)$.

By the definition of the weight filtration, we have that $\Gr_4^WH^4(X,\QQ)$ is the image of 
$$j^*:H^4(\ov X, \QQ)\to H^4(X,\QQ).$$
Therefore, the exact sequence \ref{eq:weight-four} implies that $c\in \im f$.

The map $f$ is a morphism of polarizable rational Hodge structures.
Since the category of polarizable rational Hodge structures is semisimple, there exists a splitting morphism between rational Hodge structures
\[
 h:\im f\longrightarrow\bigoplus_{i\in I}H^2(D_i,\QQ)(-1)
\]
such that 
$$f\circ h=\mathrm{id}_{\im f}.$$
Write $h(c)=(c_i(-1))_{i\in I}$, with each $c_i\in H^2(D_i,\QQ)$. Since $h$ preserves rational Hodge structure, each $c_i$
is a rational Hodge class of type $(1,1)$ on $D_i$. By the Lefschetz $(1,1)$-theorem, there exists an integer
$N_0>0$ and divisors $\gamma_i\in\CH^1(D_i)$ such that
\[
 \cl_{D_i}(\gamma_i)=N_0c_i\qquad\text{in }H^2(D_i,\QQ).
\]
Consequently, the class
\[
 \beta:=N_0\alpha'-\sum_{i\in I}(\iota_i)_*\gamma_i
 \in\CH^2(\ov X)
\]
has zero rational cohomology class and satisfies $j^*\beta=N_0\alpha$.

Using that cycle class maps commutes with proper pushforwards, we have
$$\cl_{\ov X}(\beta)=N_0\cl_{\ov X}(\alpha')-\sum_{i\in I} (\iota_i)_*\cl_{D_i}(\gamma_i)=N_0c-N_0\sum_{i\in I} (\iota_i)_*c_i=0.$$
Therefore, taking $\ov \alpha=m\beta$ for some positive integer $m$, such that $\ov\alpha$ is integrally homologically trivial satisfies the requirement.
\end{proof}
Now we propose the following result. When $X$ is projective, it is a direct consequence of \cite[Theorem 1]{BS} and \cite[Theorem 1.9]{Mur} (whose proof contains a gap and was repaired by Kahn in \cite[Theorem 1]{Kah}). 
\begin{proposition}
\label{prop:open-vanishing}
Let $X$ be a smooth, rationally connected, quasi-projective variety. Suppose that
\[
 \Gr^W_3H^3(X,\QQ)=0.
\]
Then the group $\CH^2(X)_{\homol}$ of homologically trivial cycles is torsion. In particular, the rational
cycle-class map
\[
 \cl_X:\CH^2(X)_\QQ\longrightarrow H^4(X,\QQ)
\]
is injective.
\end{proposition}

\begin{proof}
By resolution of singularities \cite{Hir}, we can choose a smooth projective
compactification $j:X\hookrightarrow\ov X$ for which
$D:=\ov X\setminus X$ is a simple normal crossings divisor. The variety
$\ov X$ is rationally connected. Write
$D=\bigcup_{i\in I}D_i$ as its irreducible components, and denote $\iota_i:D_i\hookrightarrow\ov X$ as in
Lemma~\ref{lem:lifting}.

Deligne's weight spectral sequence \cite{Del} gives an exact sequence
\begin{equation}
\label{eq:weight-three}
 \bigoplus_{i\in I}H^1(D_i,\QQ)(-1)
 \xrightarrow{\ \varphi\ }H^3(\ov X,\QQ)
 \longrightarrow\Gr^W_3H^3(X,\QQ)\longrightarrow0,
\end{equation}
where $\varphi=\sum_i(\iota_i)_*$ is the sum of Gysin pushforwards. The last term is zero by assumption, so $\varphi$ is surjective. 

Since the Gysin map preserves integral Hodge structures, we have that $\varphi$ maps
\[
 \bigoplus_{i\in I}H^1(D_i,\ZZ)_{\mathrm{free}}(-1)
 \quad\text{into}\quad H^3(\ov X,\ZZ)_{\mathrm{free}}.
\]
Here $H^*(\ov X,\ZZ)_{\mathrm{free}}:=H^*(\ov X,\ZZ)/H^*(\ov X,\ZZ)_\tor$, similarly for $D_i$.

Denote
$$J^3(\ov X):=H^3(\ov X,\CC)/(H^3(\ov X,\ZZ)_{\mathrm{free}}\oplus F^2H^3(\ov X,\CC))$$
then we have the Abel–Jacobi map \cite{VoiII}
$$\AJ_{\ov X}: \CH^2(\ov X)_{\hom}\to J^3(\ov X).$$
We find that $\varphi$ induces a homomorphism of complex tori
\begin{equation}
\label{eq:boundary-jacobian}
 \psi:\prod_{i\in I}\Pic^0(D_i)\longrightarrow J^3(\ov X).
\end{equation}
Surjectivity of $\varphi$ over $\QQ$, and hence over $\CC$, implies that $\psi$ is surjective.

Therefore, we have the following diagram
\begin{equation}
\label{eq:boundary-aj-diagram}
\begin{tikzcd}[column sep=large,row sep=large]
 \displaystyle\bigoplus_{i\in I}\CH^1(D_i)_{\homol}
   \arrow[r,"{\sum_i(\iota_i)_*}"]
   \arrow[d,"{\bigoplus_i\AJ_{D_i}}"']
 &\CH^2(\ov X)_{\homol}\arrow[d,"\AJ_{\ov X}"]\\
 \displaystyle\prod_{i\in I}\Pic^0(D_i)
   \arrow[r,"\psi"']
 &J^3(\ov X)
\end{tikzcd}
\end{equation}
and from the definition of Gysin pushforward and Abel–Jacobi map, we can check it is commutative.

For any $\alpha\in\CH^2(X)_{\homol}$, by Lemma~\ref{lem:lifting},
there exists $m>0$ and $\alpha'\in\CH^2(\ov X)_{\homol}$ such that
$j^*\alpha'=m\alpha$. By \cite[Theorem 1(ii)]{BS}, we may replace $\alpha'$ with a multiple such that $\alpha'$ is algebraically trivial. The surjectivity of $\psi$
and the commutative diagram \eqref{eq:boundary-aj-diagram} give classes
$\beta_i\in\CH^1(D_i)_{\hom}$ with
\[
 \AJ_{\ov X}(\alpha')
 =\AJ_{\ov X}\left(\sum_{i\in I}(\iota_i)_*\beta_i\right).
\]
Set
\[
 \gamma:=\alpha'-\sum_{i\in I}(\iota_i)_*\beta_i
 \in\CH^2(\ov X)_{\hom}.
\]
Then $\AJ_{\ov X}(\gamma)=0$. 

By \cite[Theorem 1(i)]{BS} and \cite[Theorem 1.9]{Mur}, the restriction of $\AJ_{\ov X}$ to $\CH^2(\ov X)_{\alg}$ is an isomorphism. Since homological equivalence and algebraic equivalence coincide on divisors, we have $\beta_i$ is algebraically trivial. Therefore, $\gamma\in \CH^2(\ov X)_{\alg}$. Hence $\gamma=0$. 

The restriction of $\gamma$ on $X$ is $m\alpha$. Hence $\alpha$ is torsion.
\end{proof}

\section{Proof of the generalized Franchetta conjecture}
\label{sec:franchetta}

In this section, we apply Proposition~\ref{prop:open-vanishing} to the universal family $\cS_g'$, and combine Bergeron–Li's result to prove Theorem \ref{thm:main}. First, we obtain the following vanishing result 
\begin{lemma}
\label{lem:universal-weight}
For every dense open subset $V\subset\cF_g'$, denote $\cS_V:=\cS_g'\times_{\cF_g'} V$, one has
\[
\Gr^W_3H^3(\cS_V,\QQ)=0.
\]
\end{lemma}
\begin{proof}
By the property of weight filtration, it is enough to
prove the lemma for $V=\cF_g'$.
Consider the Leray spectral sequence for
$\pi:\cS_g'\to\cF_g'$:
\[
E_2^{p,q}=H^p(\cF_g',R^q\pi_*\QQ)
\ \Longrightarrow\ H^{p+q}(\cS_g',\QQ).
\]
Since a K3 surface $S$ satisfies $H^1(S,\QQ)=H^3(S,\QQ)=0$, the only
terms that could contribute to $H^3(\cS_g',\QQ)$ are
\[
H^3(\cF_g',\QQ)
\qquad\text{and}\qquad
H^1(\cF_g',R^2\pi_*\QQ).
\]
Compatibility of the Leray filtration with mixed Hodge structures
shows that the only terms that could contribute to $\Gr^W_3H^3(\cS_g',\QQ)$ are 
\begin{equation}
\Gr^W_3H^3(\cF_g',\QQ)
\qquad\text{and}\qquad
\Gr^W_3H^1(\cF_g',R^2\pi_*\QQ).
\end{equation}
Here, $R^2\pi_*\QQ$ carries its natural variation of Hodge structure
of weight two.

Following the argument in \cite[Section 3]{BL}, we can take an orthogonal
arithmetic quotient $Y$ of type $\mathrm{SO}(2,19)$ and its dense open subset $\cF^l$, together with a finite surjective morphism $p:\cF^l\to \cF_g'$. Here, the variety $\cF^l$ is a dense open subset of a connected component of the moduli stack of polarized K3 surfaces of genus $g$ with a full level structure. Following \cite{BL}, the pullback $p^*R^2\pi_*\QQ$ is isomorphic to the local system corresponding to the universal family on $\cF^l$. Hence, it could be extended into an automorphic local system $\mathbf E$ on $Y$.

Apply \cite[Proposition~6.2.1]{BL} to the constant system $\mathbb Q$ and to $\mathbf E$, we have that
\begin{equation}
H^3(Y,\QQ)=0,
\qquad H^1(Y,\mathbf E)=0.
\end{equation}
We carry a weight $2$ variation of Hodge structure on $\mathbf E$. Using the property of weight filtration, we have that 
\begin{align}
W_3H^3(\cF^l,\QQ)
&\subseteq\im\bigl(H^3(Y,\QQ)
\longrightarrow H^3(\cF^l,\QQ)\bigr),\\
W_3H^1(\cF^l,p^*R^2\pi_*\QQ)
&\subseteq\im\bigl(H^1(Y,\mathbf E)
\longrightarrow H^1(\cF^l,p^*R^2\pi_*\QQ)\bigr).
\notag
\end{align}
Therefore, 
\begin{equation}
\Gr^W_3H^3(\cF^l,\QQ)=0
\qquad\text{and}\qquad
\Gr^W_3H^1(\cF^l,p^*R^2\pi_*\QQ)=0.
\end{equation}
The two maps
$$p_*:H^*(\cF^l,\QQ)\to H^*(\cF_g',\QQ)$$
$$p^*:H^*(\cF_g',\QQ)\to H^*(\cF^l,\QQ)$$
both preserve the weight filtration, and $p_*p^*=(\deg p)\operatorname{id}$, we have that 
\begin{equation}
\Gr^W_3H^3(\cF_g',\QQ)=0.
\end{equation}
Similarly, we have that
$$\Gr^W_3H^1(\cF_g',R^2\pi_*\QQ)=0.$$
Therefore, we have $\Gr_3^WH^3(\cS_g',\QQ)=0$.
\end{proof}
Now, we can prove the main result.
\begin{theorem}
\label{thm:franchetta}
Suppose that $g\ge3$ and that $\cS_g'$ is rationally connected. Then Conjecture \ref{Generalized Franchetta Conjecture} holds for genus $g$.
\end{theorem}

\begin{proof}
Fix a cycle $\alpha\in \CH^2(\cS_g')$. By a standard spread-out argument \cite[Section~1.1.2]{Voi}, it suffices to prove the result for $\alpha$ for a general closed point $x\in \cF_g'$. In addition, since the Chow group of a complex K3 surface is torsion-free \cite{Roi}, it suffices to prove the result up to torsion.

By Bergeron--Li's cohomological generalized Franchetta theorem (Theorem \ref{thm: cohomological GFC}), there exists a rational number $t$ such that the
cohomology class of
\[
 \alpha-tc_2(T_\pi)
\]
is supported on the inverse image of a finite union of
Noether--Lefschetz divisors on $\cF_g'$ under $\pi:\cS_g'\to \cF_g'$. Consequently, there exists a dense open subset
$V\subset\cF_g'$ such that
\begin{equation}
\label{eq:cohomological-relation}
 \cl_{\cS_V}\bigl((\alpha-tc_2(T_\pi))|_{\cS_V}\bigr)=0
 \in H^4(\cS_V,\QQ).
\end{equation}

Take an integer $m>0$ such that $mt$ is an integer. Since $\cS_V$ is rationally connected, combining Lemma \ref{lem:universal-weight} and Proposition \ref{prop:open-vanishing}, we have that
$$m(\alpha-tc_2(T_\pi))|_{\cS_V}\in \CH^2(\cS_V)$$
is torsion.

From the property of Beauville–Voisin class \cite{BV}, for any closed point $x\in V$, we have that $c_2(S_x)=24o_{S_x}$. Therefore, the restriction of $\alpha$ of $S_x$ is a multiple of Beauville–Voisin class $o_{S_x}$ up to torsion. Hence we proved the theorem.

\end{proof}


\begin{thebibliography}{99}

\bibitem{AC87}
E. Arbarello and M. Cornalba, The Picard groups of the moduli spaces
of curves, Topology {\bf 26} (1987), no.~2, 153--171; MR0895568

\bibitem{Bar}
I. Barros, Geometry of the moduli space of $n$-pointed $K3$ surfaces
of genus 11, Bull. Lond. Math. Soc. {\bf 50} (2018), no.~6,
1071--1084; MR3891944

\bibitem{BV}
A. Beauville and C. Voisin, On the Chow ring of a $K3$ surface,
J. Algebraic Geom. {\bf 13} (2004), no.~3, 417--426; MR2047674

\bibitem{BL}
N. Bergeron and Z. Li, Tautological classes on moduli spaces of
hyper-K\"ahler manifolds, Duke Math. J. {\bf 168} (2019), no.~7,
1179--1230; MR3953432

\bibitem{BS}
S.~J. Bloch and V. Srinivas, Remarks on correspondences and algebraic
cycles, Amer. J. Math. {\bf 105} (1983), no.~5, 1235--1253; MR0714776

\bibitem{Del}
P. Deligne, Th\'eorie de Hodge. II, Inst. Hautes \'Etudes Sci. Publ.
Math. No.~40 (1971), 5--57; MR0498551

\bibitem{FV14}
G. Farkas and A. Verra, The universal $K3$ surface of genus 14 via
cubic fourfolds, J. Math. Pures Appl. (9) {\bf 111} (2018), 1--20; MR3760746

\bibitem{FV22}
G. Farkas and A. Verra, The unirationality of the moduli space of
$K3$ surfaces of genus $22$, Math. Ann. {\bf 380} (2021), no.~3--4,
953--973; MR4297179

\bibitem{FL}
L. Fu and R.~C. Laterveer, Special cubic four-folds, $K3$ surfaces,
and the Franchetta property, Int. Math. Res. Not. IMRN {\bf 2023},
no.~10, 8872--8902; MR4589089

\bibitem{Hir}
H. Hironaka, Resolution of singularities of an algebraic variety over
a field of characteristic zero. I, II, Ann. of Math. (2) {\bf 79}
(1964), 109--203 and 205--326; MR0199184

\bibitem{Rethlas}
H. Ju, G. Gao, J. Jiang, B. Wu, Z. Sun, S. Liu, L. Chen,
Y. Wang, Y. Wang, Z. Wang, W. He, P. Wu, L. Xiao,
R. Liu, B. Dai and B. Dong, Automated conjecture resolution with
formal verification, preprint,
\url{https://arxiv.org/abs/2604.03789}.

\bibitem{Kah}
B. Kahn, On the universal regular homomorphism in codimension 2,
Ann. Inst. Fourier (Grenoble) {\bf 71} (2021), no.~2, 843--848;
MR4353922

\bibitem{Lu}
Y. Lu, On O'Grady's generalized Franchetta conjecture for genus $11$
$K3$ surfaces, preprint,
\url{https://arxiv.org/abs/2511.16875}.

\bibitem{Muk}
S. Mukai, Curves and $K3$ surfaces of genus eleven, in {\it Moduli of
vector bundles (Sanda, 1994; Kyoto, 1994)}, 189--197, Lecture Notes in
Pure and Appl. Math., 179, Dekker, New York, 1996; MR1397987

\bibitem{Mur}
J.~P. Murre, Applications of algebraic $K$-theory to the theory of
algebraic cycles, in {\it Algebraic geometry, Sitges (Barcelona), 1983},
216--261, Lecture Notes in Math., 1124, Springer, Berlin, 1985;
MR0805336

\bibitem{OG}
K.~G. O'Grady, Moduli of sheaves and the Chow group of $K3$ surfaces,
J. Math. Pures Appl. (9) {\bf 100} (2013), no.~5, 701--718; MR3115830

\bibitem{PSY}
N. Pavic, J. Shen and Q. Yin, On O'Grady's generalized Franchetta
conjecture, Int. Math. Res. Not. IMRN {\bf 2017}, no.~16, 4971--4983;
MR3687122

\bibitem{Roi}
A.~A. Ro\u{\i}tman, The torsion of the group of $0$-cycles modulo
rational equivalence, Ann. of Math. (2) {\bf 111} (1980), no.~3,
553--569; MR0577137

\bibitem{VoiII}
C. Voisin, {\it Hodge theory and complex algebraic geometry. II},
Cambridge Studies in Advanced Mathematics, 77, Cambridge Univ. Press,
Cambridge, 2003; MR1997577

\bibitem{Voi}
C. Voisin, {\it Chow rings, decomposition of the diagonal, and the
topology of families}, Annals of Mathematics Studies, 187, Princeton
Univ. Press, Princeton, NJ, 2014; MR3186044

\end{thebibliography}
\end{document}